\documentclass{amsart}
\usepackage{graphicx}
\usepackage[nocompress]{cite}
\usepackage{tikz-cd}
\usepackage{subcaption}
\usepackage{mathrsfs}
\usepackage{amsopn,amssymb,amsmath}
\usepackage{url}
\usepackage{subcaption}
\usepackage{algorithmic}
\usepackage{algorithm}
\usepackage{mleftright}
\theoremstyle{definition}
\newtheorem{theorem}{Theorem}[section]
\newtheorem{prop}[theorem]{Proposition}
\newtheorem{defn}[theorem]{Definition}

\newtheorem{eg}[theorem]{Example}
\newtheorem{lemma}[theorem]{Lemma}
\newtheorem{remark}[theorem]{Remark}
\numberwithin{equation}{section}

\DeclareMathOperator{\Ima}{Im}

\begin{document}
\title{Weighted Persistent Homology}

\author{Shiquan Ren}
\address{\textit{a} School of Mathematics and Computer Science, Guangdong Ocean University, 1 Haida Road, Zhanjiang, China, 524088.\newline
\textit{b} Department of Mathematics,   National University of Singapore, Singapore, 119076.}
\email{sren@u.nus.edu}
\thanks{The project was supported in part by the Singapore Ministry of Education research grant (AcRF Tier 1 WBS No.~R-146-000-222-112). The first author was supported in part by the National Research Foundation, Prime Minister's Office, Singapore under its Campus for Research Excellence and Technological Enterprise (CREATE) programme. The second author was supported in part by the President's Graduate Fellowship of National University of Singapore. The third author was supported by a grant (No.~11329101) of NSFC of China.}

\author{Chengyuan Wu$^*$}
\address{Department of Mathematics, National University of Singapore, Singapore 119076}
\email{wuchengyuan@u.nus.edu}
\thanks{$^*$Corresponding author}

\author{Jie Wu}
\address{Department of Mathematics, National University of Singapore, Singapore 119076}
\email{matwuj@nus.edu.sg}


\begin{abstract}
In this paper we develop the theory of weighted persistent homology. In 1990, Robert J.\ Dawson was the first to study in depth the homology of weighted simplicial complexes. We generalize the definitions of weighted simplicial complex and the homology of weighted simplicial complex to allow weights in an integral domain $R$. Then we study the resulting weighted persistent homology. We show that weighted persistent homology can tell apart filtrations that ordinary persistent homology does not distinguish. For example, if there is a point considered as special, weighted persistent homology can tell when a cycle containing the point is formed or has disappeared. 

\end{abstract}

\maketitle
\section{Introduction}

In topological data analysis, point cloud data refers to a finite set $X$ of points  in the Euclidean space $\mathbb{R}^n$, and the computation of persistent homology usually 
begins with the point cloud data $X$. 
In the 
classical approach of the persistent homology of $X$, each point in $X$ plays an equally important role, or in other words, each point has the same weight (cf. \cite{Zomorodian2005}). Then  $X$ is converted into a simplicial complex, for example, the \v{C}ech complex and the Vietoris-Rips complex (cf. \cite[Chap.~III]{Edelsbrunner2010}).

In this paper, we consider the situation that different points in $X$ may have different importance. Our point cloud data $X$ is  weighted, that is, each point 
in $X$ has a  weight. 
Some practical examples  where it is 
useful to consider weighted point cloud data are described in Section~\ref{subs-m}.

Our approach is to weight the boundary map. This is different from existing techniques of introducing weights to persistent homology. For instance in the paper \cite{Petri2013} by Petri, Scolamiero, Donato, and Vaccarino, weights are introduced via the \emph{weight rank clique filtration} with a thresholding of weights, where at each step $t$ the thresholded graph with links of weight larger than a threshold $\epsilon_t$. In the paper \cite{edelsbrunner2012persistent} by Edelsbrunner and Morozov, the weight of edges is also used to construct a filtration. The theory of weighted simplicial complexes we use is also significantly different from the theory of \emph{weighted alpha shapes} \cite{edelsbrunner1992weighted} by H.\ Edelsbrunner, which are polytopes uniquely determined by points, their weights, and a parameter $\alpha\in\mathbb{R}$ that controls the level of detail.

In his thesis \cite{curry2013sheaves}, J. Curry utilizes the barcode descriptor from persistent homology to interpret cellular cosheaf homology in terms of Borel-Moore homology of the barcode. In \cite[p.~244]{curry2013sheaves}, it is briefly mentioned that for applications, cosheaves should allow us to weight different models of the real world. In a subsequent work \cite{curry2016discrete} by J.\ Curry, R.\ Ghrist and V.\ Nanda, it is shown how sheaves and sheaf cohomology are powerful tools in computational topology, greatly generalizing persistent homology. An algorithm for simplifying the computation of cellular sheaf cohomology via (discrete) Morse-theoretic techniques is included in \cite{curry2016discrete}. In the recent paper \cite{kashiwara2017persistent}, M.\ Kashiwara and P.\ Schapira show that many results in persistent homology can be interpreted in the language of microlocal sheaf theory. We note that in \cite[p.~8]{kashiwara2017persistent}, a notion of weight is being used, where the closed ball $B(s;t)$ is being replaced by $B(s;\rho(s)t)$, where $\rho(s)\in\mathbb{R}_{\geq 0}$ is the weight. This notion of weight is more geometrical, which differs from our more algebraic approach of weighting the boundary operator.

In the seminal paper \cite{carlsson2009theory} by Carlsson and Zomorodian, the theory of multidimensional persistence of multidimensional filtrations is developed. In a subsequent work \cite{carlsson2009computing} by Carlsson, Singh and Zomorodian, a polynomial time algorithm for computing multidimensional persistence is presented. In \cite{cerri2013betti}, the authors Cerri, Fabio, Ferri, Frosini and Landi show that Betti numbers in multidimensional persistent homology are stable functions, in the sense that small changes of the vector-valued filtering functions imply only small changes of persistent Betti numbers functions. In \cite{xia2015multidimensional}, K. Xia and G. Wei introduce two families of multidimensional persistence, namely pseudomultidimensional persistence and multiscale multidimensional persistence, and apply them to analyze biomolecular data. The utility and robustness of the proposed topological methods are effectively demonstrated via protein folding, protein flexibility analysis, and various other applications. In \cite[p.~1509]{xia2015multidimensional}, a particle type-dependent weight function $w_j$ is introduced in the definition of the atomic rigidity index $\mu_i$. The atomic rigidity index $\mu_i$ can be generalized to a position ($\mathbf{r}$)-dependent rigidity density $\mu(\mathbf{r})$. Subsequently \cite[p.~1512]{xia2015multidimensional}, filtration is performed over the density $\mu(\mathbf{r})$. 

The main aim of our paper is to construct weighted persistent homology to study the topology  of weighted point cloud data. A weighted simplicial complex is a simplicial complex where each simplex is assigned with a weight. We  convert a weighted point cloud data $X$ into a weighted simplicial complex. 
In \cite{Dawson1990}, Robert J.\ Dawson was the first to study in depth the homology of weighted simplicial complexes. We use an adaptation of \cite{Dawson1990} to compute the homology of weighted simplicial complexes. In \cite{Dawson1990}, the weights take values in the set of non-negative integers $\{0,1,2,\dots\}$. We generalize \cite{Dawson1990} such that the weights can take values in any integral domain $R$ with multiplicative identity $1_R$. Finally, we study and analyze the weighted persistent homology of filtered weighted simplicial complexes.

\section{Background}

In this section, we review some background knowledge and give some preliminary definitions. We give some examples of weighted cloud data in Subsection~\ref{subs-m}
. We review the definitions of simplicial complexes in Subsection~\ref{subs2.2} and review some properties of rings and integral domains in Subsection~\ref{subs2.3}. We give the formal definitions of weighted point cloud data and weighted simplicial complex in Subsection~\ref{subs2.4}.

\subsection{
Examples of weighted cloud data}\label{subs-m}
As the motivation of this paper, we describe 
some practical problems with weight function on data. We look at typical applications of persistent homology, and consider the situation that data points may not be equally important. When some data points may be more important than others, mathematically it requires a weight function to give the difference between points. 

In the field of computer vision, Carlsson et al.\ \cite{Carlsson2008} develops a framework to use persistent homology to analyze natural images such as digital photographs. The natural image may be viewed as a vector in a very high-dimensional space $\mathcal{P}$. In the paper, the dimension of $\mathcal{P}$ is the number of pixels in the format used by the camera and the image is associated to the vector whose coordinates are grey scale values of the pixels. In certain scenarios, such as color detection in computer vision \cite{Pedreschi2006,Lu2000,Comaniciu1997}, each pixel may play different roles depending on its color. In this case, each pixel can then be given a different weight depending on its color. More generally, pixels in images can be weighted depending on its wavelength in the electromagnetic spectrum which includes infrared and ultraviolet light.

In the paper \cite{Carstens2013}, persistent homology is used to study collaboration networks, which measures how scientists collaborate on papers. In the collaboration network, there is a connection between two scientists if they are coauthors on at least one paper. Depending on the purpose of research, weights can be used to differentiate different groups of scientists, for example PhD students, postdoctoral researchers and professors, or researchers in different fields.

Lee et al.\ \cite{Lee2012} proposed a new framework for modeling brain connectivity using persistent homology. The connectivity of the human brain, also known as human connectome, is usually represented as a graph consisting of nodes and edges connecting the nodes. In this scenario, different weights could be assigned to different neurons in different parts of the brain, for example left/right brain, frontal lobe or temporal lobe.

There are many different ways to define the theory of weighted persistent homology (cf.\ \cite{Horak2013,Petri2013,Dawson1990}), and our definition is not unique. We will show that our definition satisfies some nice properties, including some properties related to category theory \cite{Lane1978} which is an important part of modern mathematics.


In this section we review the mathematical background necessary for our work. We assume all rings have the multiplicative identity $1$. First we define the concept of weighted point cloud data (WPCD). Then, similar to the unweighted case, we can convert the WPCD to a simplicial complex, using either the \v{C}ech complex or the Vietoris-Rips complex. Then, we define a weight function for the simplices so as to compute the weighted simplicial homology.
\subsection{Simplicial Complexes}\label{subs2.2}

The following definition of simplicial complexes can be found in \cite[p. 107]{Hatcher2002}. 
An \emph{(abstract) simplicial complex} is a collection $K$ of nonempty finite sets,  called \emph{(abstract) simplices}, such that if $\sigma\in K$, then every nonempty subset of $\sigma$ is in $K$. Let $K$ be a simplicial complex and let $\sigma\in K$. 
An element $v$ of  $\sigma$ is called a \emph{vertex}
, and any nonempty subset of $\sigma$ is called a \emph{face}
. For convenience, we do not distinguish between a vertex $v$ and the corresponding face $\{v\}$.

The definition of orientations of simplices is given in \cite[p. 105]{Hatcher2002}.  Let  $\sigma=
\{v_0,v_1,\dots,v_n\}$ be a simplex of a simplicial complex $K$. 
An orientation of $\sigma$ is given by an ordering of its vertices  $v_0,v_1,\dots,v_n$, with the rule that two orderings define the same orientation if and only if they differ by an even permutation. An oriented simplex $\sigma$ is written as $[v_0,v_1,\cdots,v_n]$.  


Let $\{x_\alpha\}$ be a set of points in the Euclidean space $\mathbb{R}^n$.  Let $\epsilon>0$.    The \emph{\v{C}ech complex}, 
denoted as $\mathcal{C}_\epsilon$, is the abstract simplicial complex where  $k+1$ vertices span a $k$-simplex if and only if  the $k+1$ corresponding closed $\epsilon/2$-ball neighborhoods of the vertices  have nonempty intersection (cf. \cite[p. 72]{Edelsbrunner2010}).
The \emph{Vietoris-Rips complex}, denoted as $\mathcal{R}_\epsilon$, is the abstract simplicial complex where  $k+1$ vertices span a $k$-simplex if and only if the distance between any pair of the $k+1$ vertices  is at most $\epsilon$ (cf. \cite[p.~74]{Edelsbrunner2010}).

\subsection{Rings}\label{subs2.3}
Throughout this section, we let $R$ to be a commutative ring with multiplicative identity. A nonzero element $a\in R$ is said to \emph{divide} an element $b\in R$ (denoted $a\mid b$) if there exists $x\in R$ such that $ax=b$.
A nonzero element $a$ in a ring $R$ is called a 
\emph{zero divisor} if there exists a nonzero $x\in R$ such that $ax=0$. 
A commutative ring $R$ with $1_R\neq 0$ and no zero divisors is called an \emph{integral domain} (cf. \cite[p.~116]{Hungerford}). 


Let $R$ be an integral domain.   Let $S$ be the set of all nonzero elements in $R$.  Then  we can construct the \emph{quotient field} $S^{-1}R$ (cf.  \cite[p.~142]{Hungerford}).

\begin{prop}{\cite[p.~144]{Hungerford}}
\label{embedq}
The map $\varphi_s:R\to S^{-1}R$ given by $r\mapsto rs/s$ (for any $s\in S$) is a monomorphism. Hence, the integral domain $R$ can be embedded in its quotient field.
\end{prop}
\begin{remark}
\label{identifyfrac}
Due to Proposition \ref{embedq}, we may identify $rs/s \in S^{-1}R$ with $r\in R$. We denote this as $\varphi_s^{-1}(rs/s)=r$, or simply $rs/s=r$ if there is no danger of confusion.
\end{remark}

\subsection{Weighted Simplicial Complexes}\label{subs2.4}

In the following definitions, we generalize Robert J.\ Dawson's paper \cite{Dawson1990} and define \emph{weighted point cloud data} and \emph{weighted simplicial complexes}, with weights in rings.

\begin{defn}[Weighted point cloud data]
Let $n$ be a positive integer. The \emph{point cloud data} $X$ in $\mathbb{R}^n$ is a finite subset of $\mathbb{R}^n$. Given some point cloud data $X$, 
a \emph{weight} on $X$ is a function $w_0: X\to R$, where $R$ is a commutative ring. The pair $(X,w_0)$ is called \emph{weighted point cloud data}, or \emph{WPCD} for short.
\end{defn}
Next, in Definition \ref{wscdef} we generalize the definition of \emph{weighted simplicial complex} in \cite[p.~229]{Dawson1990} to allow for weights in a commutative ring.
\begin{defn}[cf.\ {\cite[p.~229]{Dawson1990}}] 
\label{wscdef}
A \emph{weighted simplicial complex} (or \emph{WSC} for short) is  a pair $(K,w)$ consisting of a simplicial complex $K$ and a  function $w: K\to R$, where $R$ is a commutative ring, such that for any $\sigma_1, \sigma_2\in K$ with $\sigma_1\subseteq \sigma_2$, we have $w(\sigma_1)\mid w(\sigma_2)$.
\end{defn}
Given any weighted point cloud data $(X,w_0)$, we allow flexible definitions of extending the weight function $w_0$ to all higher-dimensional simplices, where the only condition to be satisfied is the divisibility condition in Definition \ref{wscdef}. One such definition is what we call the \emph{product weighting}.
\begin{defn}[Product weighting]
\label{product}
Let $(X,w_0)$ be a weighted point cloud data, with weight function $w_0: X\to R$ (where $R$ is a commutative ring). Let $K$ be a simplicial complex 
whose set of vertices is $X$. We define a weight function $w: K\to R$ by 
\begin{align}\label{e1}
w(\sigma)=\prod_{i=0}^k w_0(v_i).
\end{align}
 where $\sigma=[v_0,v_1,\dots,v_k]$ is a $k$-simplex of $K$. We call $w$ defined as such the \emph{product weighting}.
\end{defn}
\begin{prop}
\label{prop1}
Let $(X,w_0)$ be a weighted point cloud data. Let $w$ be the product weighting defined in Definition \ref{product}. Then the followings hold:
\begin{enumerate}
\item The restriction of $w$ to the vertices of $K$ is $w_0$.
\item For any $\sigma_1,\sigma_2\in K$, if $\sigma_1\subseteq\sigma_2$, then $w(\sigma_1)\mid w(\sigma_2)$.
\end{enumerate}
\end{prop}

\begin{proof}
Firstly, if $\sigma=[v_0]$ is a vertex of $K$ (0-simplex), then $w(\sigma)=w_0(v_0)$ by (\ref{e1}).

For the second assertion, suppose $\sigma_1\subseteq\sigma_2$, where $\sigma_1=[v_0,\dots,v_k]$ and $\sigma_2=[v_0,\dots,v_k,\dots,v_l]$. Then $w(\sigma_2)=w(\sigma_1)\cdot\prod_{i=k+1}^l w_0(v_i)$.
\end{proof}

For commutative rings such that every two elements have a LCM (for instance UFDs), we can use the economical weighting in \cite[p.~231]{Dawson1990} instead, where the weight of any simplex is the LCM of the weights of its faces.


\section{Properties of Weighted Simplicial Complexes}
In this section, we prove some properties of weighted simplicial complexes. We consider the case where $R$ is a commutative ring with 1.
We now consider subcomplexes given by the preimage of the weight function with values in ideals. This may have the meaning to take out partial data according to the values of the weight function.
\begin{lemma}
\label{lemma2}
Let $I$ be an ideal of a commutative ring $R$. Let $(K,w)$ be a WSC, where $w: K\to R$ is a weight function. Let $w^{-1}(I)$ denote the preimage of $I$ under $w$. If $\sigma\in w^{-1}(I)$, then for all simplices $\tau$ containing $\sigma$, we have $\tau\in w^{-1}(I)$.
\end{lemma}
\begin{proof}
Let $\sigma\in w^{-1}(I)$, i.e.\ $w(\sigma)\in I$. By Definition \ref{wscdef}, for $\sigma\subseteq\tau$ we have $w(\sigma)\mid w(\tau)$. Hence $w(\tau)=w(\sigma)x$ for some $x\in R$. Since $I$ is an ideal, thus $w(\tau)\in I$.
\end{proof}
\begin{theorem}
\label{preimageideal}
Let $I$ be an ideal of a commutative ring $R$. Let $(K,w)$ be a WSC, where $w:K\to R$ is a weight function. Then $K\setminus w^{-1}(I)$ is a simplicial subcomplex of $K$.
\end{theorem}
\begin{proof}
If $K\setminus w^{-1}(I)=\emptyset$, then it is the empty subcomplex of $K$.

Otherwise, let $\tau\in K\setminus w^{-1}(I)$. Let $\sigma$ be a nonempty subset of $\tau$. Suppose to the contrary $\sigma\in w^{-1}(I)$. Then by Lemma \ref{lemma2}, we have $\tau\in w^{-1}(I)$, which is a contradiction. Hence $\sigma\in K\setminus w^{-1}(I)$, so we have proved that $K\setminus w^{-1}(I)$ is a simplicial complex.
\end{proof}
\begin{prop}
Let $I$, $J$ be ideals of a commutative ring $R$. Let $(K,w)$ be a WSC. Then
\begin{equation}
\label{intersectideal}
K\setminus w^{-1}(I\cap J)=(K\setminus w^{-1}(I))\cup(K\setminus w^{-1}(J))
\end{equation}
is a simplicial subcomplex of $K$.
\end{prop}
\begin{proof}
We have that 
\begin{align*}
\sigma\in K\setminus w^{-1}(I\cap J)&\iff w(\sigma)\notin I\cap J\\
&\iff w(\sigma)\notin I\ \text{or}\ w(\sigma)\notin J\\
&\iff \sigma\in(K\setminus w^{-1}(I))\cup(K\setminus w^{-1}(J)).
\end{align*}
Hence Equation \ref{intersectideal} holds.

Since $I$, $J$ are ideals, by Theorem \ref{preimageideal} both $K\setminus w^{-1}(I)$ and $K\setminus w^{-1}(J)$ are simplicial subcomplexes of $K$ and so is their union. Alternatively, we can apply Theorem \ref{preimageideal} to the ideal $I\cap J$ to conclude that $K\setminus w^{-1}(I\cap J)$ is a simplicial subcomplex of $K$.
\end{proof}
\subsection{Categorical Properties of WSC}
Let $K$ and $L$ be simplicial complexes. A map $f:K\to L$ is called a \emph{simplicial map} if it sends each simplex of $K$ to a simplex of $L$ by a linear map taking vertices to vertices. That is, if the vertices $v_0,\dots,v_n$ of $K$ span a simplex of $K$, the points $f(v_0),\dots,f(v_n)$ (not necessarily distinct) span a simplex of $L$.

Next, we will use some terminology from Category Theory. We recommend the book by Mac Lane \cite{Lane1978} for an introduction to the subject. The categorical properties of WSCs have been studied in \cite{Dawson1990}. Here, we mainly show that it easily generalizes to the case where weights lie in a ring, and write it in greater detail.

\begin{defn}[{\cite[p.~13]{Lane1978}}]
Let $C$ and $B$ be categories. A \emph{functor} $T:C\to B$ with domain $C$ and codomain $B$ consists of two suitably related functions: The object function $T$, which assigns to each object $c$ of $C$ an object $Tc$ of $B$ and the arrow function (also written as $T$) which assigns to each arrow $f:c\to c'$ of $C$ an arrow $Tf: Tc\to Tc'$ of $B$, such that \[T(1_c)=1_{Tc},\qquad T(g\circ f)=Tg\circ Tf,\] where the latter holds whenever the composite $g\circ f$ is defined in $C$.
\end{defn}

In \cite[p. 229]{Dawson1990}, morphisms of weighted simplicial complexes with integral weights   have been studied. In the following definition, we generalize \cite{Dawson1990} and define morphisms of weighted simplicial complexes with weights in general commutative rings.
\begin{defn}[cf.\ {\cite[p.~229]{Dawson1990}}]
\label{WSCmorphism}
Let $(K,w_K)$ and $(L,w_L)$ be WSCs. A \emph{morphism of WSCs} is a simplicial map $f: K\to L$ such that $w_L(f(\sigma))\mid w_K(\sigma)$ for all $\sigma\in K$. These form the morphisms of a category \textbf{WSC}. We may omit the subscripts in $w_K$ and $w_L$, for instance writing $w(f(\sigma))\mid w(\sigma)$, if there is no danger of confusion.
\end{defn}

The next example generalizes \cite[p. 229]{Dawson1990}. 
\begin{eg}[cf.\ {\cite[p.~229]{Dawson1990}}]
For any simplicial complex $K$ and every $a\in R$, there is a WSC $(K,a)$ in which every simplex (in particular every vertex) has weight $a$. We call this construction a \emph{constant weighting}.
\end{eg}
Let \textbf{SC} denote the category of simplicial complexes.
\begin{prop}[cf.\ {\cite[p.~229]{Dawson1990}}]
Constant weightings are functorial: Let $T:\textbf{SC}\to\textbf{WSC}$ be defined by $TK=(K,a)$ for each simplicial complex $K\in\textbf{SC}$ and $Tf=f$ for each simplicial map $f\in\textbf{SC}$. Then $T$ is a functor.
\end{prop}
\begin{proof}
Straightforward verification. Note that the condition $w(f(\sigma))\mid w(\sigma)$ in Definition \ref{WSCmorphism} is trivially satisfied since $a\mid a$ for all $a\in R$.
\end{proof}
\begin{defn}[{\cite[p.~80]{Lane1978}}]
Let $A$ and $X$ be categories. An \emph{adjunction} from $X$ to $A$ is a triple $\langle F,G,\varphi\rangle$, where $F:X \to A$ and $G: A\to X$ are functors and $\varphi$ is a function which assigns to each pair of objects $x\in X$, $a\in A$ a bijection of sets \[\varphi=\varphi_{x,a}: A(Fx,a)\cong X(x,Ga)\] which is natural in $x$ and $a$.
An adjunction may also be described directly in terms of arrows. It is a bijection which assigns to each arrow $f: Fx\to a$ an arrow $\varphi f=\text{rad}\,f:x\to Ga$, the \emph{right adjunct} of $f$, such that \[\varphi(k\circ f)=Gk\circ\varphi f,\qquad \varphi(f\circ Fh)=\varphi f\circ h\] hold for all $f$ and all arrows $h: x'\to x$ and $k: a\to a'$. Given such an adjunction, the functor $F$ is said to be a \emph{left adjoint} for $G$, while $G$ is called a \emph{right adjoint} for $F$.
\end{defn}
One reason why we generalize the weights to take values in rings with 1 is to keep the following nice proposition true.
\begin{prop}[cf.\ {\cite[p.~229]{Dawson1990}}]
\label{prop:rladjoint}
The constant weighting functor $T_1:=(-,1_R)$ and $T_0:=(-,0_R)$ are respectively right and left adjoint to the forgetful functor $U$ from \textbf{WSC} to the category \textbf{SC} of simplicial complexes.
\end{prop}
\begin{proof}
Let $\varphi$ be a bijection that assigns to each arrow $f: U(K,w)\to L$ an arrow $\varphi f: (K,w)\to (L,1)$, where $\varphi f(\sigma)=f(\sigma)$. The key point is that the condition for WSC morphism (Def.\ \ref{WSCmorphism}), namely $1\mid w(\sigma)$, always holds for all $\sigma\in (K,w)$. Then for all arrows $h:(K,w)\to (K',w')$ and $k: L\to L'$, we have $\varphi(k\circ f)=k\circ f=Uk\circ\varphi f$ and $\varphi(f\circ Uh)=f\circ Uh=\varphi f\circ h$. Thus $T_1$ is the right adjoint for $U$.

Let $\psi$ be a bijection that assigns to each arrow $f': (K,0)\to (L,w')$ an arrow $\psi f': K\to U(L,w')$, where $\psi f'(\sigma)=f'(\sigma)$. The key point is that $w'(f'(\sigma))\mid 0$ always holds for all $\sigma\in (K,0)$. Similarly, we can conclude that $T_0$ is the left adjoint for $U$.
\end{proof}
\section{Homology of Weighted Simplicial Complexes}
\label{homology}
In this section, we let $R$ be an integral domain, in order to form the field of fractions (also known as quotient field) which is needed for our purposes. 
\subsection{Chain complex}
A \emph{chain complex} $(C_\bullet,\partial_\bullet)$ is a sequence of abelian groups or modules $\dots, C_2, C_1, C_0, C_{-1}, C_{-2},\dots$ connected by homomorphisms (called boundary homomorphisms) $\partial_n: C_n\to C_{n-1}$, such that $\partial_n\circ\partial_{n+1}=0$ for each $n$. A chain complex is usually written out as: \[\dots\to C_{n+1}\xrightarrow{\partial_{n+1}}C_n\xrightarrow{\partial_n}C_{n-1}\to\dots\to C_1\xrightarrow{\partial_1}C_0\xrightarrow{\partial_0}C_{-1}\xrightarrow{\partial_{-1}}C_{-2}\to\dots\]
A \emph{chain map} $f$ between two chain complexes $(A_\bullet, \partial_{A,\bullet})$ and $(B_\bullet, \partial_{B,\bullet})$ is a sequence $f_\bullet$ of module homomorphisms $f_n: A_n\to B_n$ for each $n$ that commutes with the boundary homomorphisms on the two chain complexes: \[\partial_{B, n}\circ f_n=f_{n-1}\circ \partial_{A,n}.\]
\[
\begin{tikzcd}
\dots\arrow[r] &A_{n+1}\arrow[d,"f_{n+1}"]\arrow[r,"\partial_{A,n+1}"] &A_{n}\arrow[d,"f_n"]\arrow[r,"\partial_{A,n}"] &A_{n-1}\arrow[d,"f_{n-1}"]\arrow[r] &\dots\\
\dots\arrow[r] &B_{n+1}\arrow[r,"\partial_{B,n+1}"] &B_n\arrow[r,"\partial_{B,n}"] &B_{n-1}\arrow[r]&\dots
\end{tikzcd}
\]
\subsection{Homology Groups}
For a topological space $X$ and a chain complex $C(X)$, the \emph{$n$th homology group} of $X$ is $H_n(X):=\ker(\partial_n)/\Ima(\partial_{n+1})$. Elements of $B_n(X):=\Ima(\partial_{n+1})$ are called \emph{boundaries} and elements of $Z_n(X):=\ker(\partial_n)$ are called \emph{cycles}.
\begin{prop}
\label{chainmapinduce}
A chain map $f_\bullet$ between chain complexes $(A_\bullet, \partial_{A, \bullet})$ and $(B_\bullet, \partial_{B,\bullet})$ induces homomorphisms between the homology groups of the two complexes.
\end{prop}
\begin{proof}
The relation $\partial f=f\partial$ implies that $f$ takes cycles to cycles since $\partial\alpha=0$ implies $\partial(f\alpha)=f(\partial\alpha)=0$. Also $f$ takes boundaries to boundaries since $f(\partial\beta)=\partial(f\beta)$.

For $\beta\in\Ima\partial_{A,n+1}$, we have $\pi_{B,n}f_n(\beta)=\Ima\partial_{B,n+1}$. Therefore $\Ima\partial_{A,n+1}\subseteq\ker(\pi_{B,n}\circ f_n)$. By the universal property of quotient groups, there exists a unique homomorphism $(f_n)_*$ such that the following diagram commutes.
\[
\begin{tikzcd}[column sep=tiny]
\ker\partial_{A,n}\arrow[r,"f_n"]\arrow[rd,"\pi_{A,n}"] &\ker\partial_{B,n}\arrow[r,"\pi_{B,n}"] &H_n(B_\bullet)=\ker\partial_{B,n}/\Ima\partial_{B,n+1}\\
&H_n(A_\bullet)=\ker\partial_{A,n}/\Ima\partial_{A,n+1}\arrow[ur,"(f_n)_*"]
\end{tikzcd}
\]

Hence $f_\bullet$ induces a homomorphism $(f_\bullet)_*: H_\bullet (A_\bullet)\to H_\bullet (B_\bullet)$.
\end{proof}
\begin{defn}
\label{chaingroup}
Let $C_n(K,w)$ (or simply $C_n(K)$ where unambiguous) be the free $R$-module with basis the $n$-simplices of $K$ with nonzero weight. Elements of $C_n(K)$, called $n$\emph{-chains}, are finite formal sums $\sum_\alpha n_\alpha\sigma_\alpha$ with coefficients $n_\alpha\in R$ and $\sigma_\alpha\in K$.
\end{defn}
\begin{defn}
\label{chain}
Given a simplicial map $f: K\to L$, the induced homomorphism $f_\sharp: C_n(K)\to C_n(L)$ is defined on the generators of $C_n(K)$ (and extended linearly) as follows. For $\sigma=[v_0,v_1,\dots,v_n]\in C_n(K)$, we define
\begin{equation}f_\sharp(\sigma)=\begin{cases}\frac{w(\sigma)}{w(f(\sigma))}f(\sigma)&\text{if $f(v_0),\dots,f(v_n)$ are distinct,}\\
0&\text{otherwise,}
\end{cases}
\end{equation}
where $\frac{w(\sigma)}{w(f(\sigma))}\in S^{-1}R$ is identified with the corresponding element in $R$ as described in Remark \ref{identifyfrac}.

Note that this is well-defined since if $w(\sigma)\neq 0$, then $w(f(\sigma))\mid w(\sigma)$ in Definition \ref{WSCmorphism} implies $w(f(\sigma))\neq 0$. So $\frac{w(\sigma)}{w(f(\sigma))}\in S^{-1}R$. Furthermore, $\frac{w(\sigma)}{w(f(\sigma))}=\frac{xw(f(\sigma))}{w(f(\sigma))}$ for some $x\in R$, so that $\frac{w(\sigma)}{w(f(\sigma))}=x\in R$.
\end{defn}
\begin{defn}[cf.\ {\cite[p.~234]{Dawson1990}}]
\label{boundary}
The \emph{weighted boundary map} $\partial_n: C_n(K)\to C_{n-1}(K)$ is the map: \[\partial_n(\sigma)=\sum_{i=0}^n\frac{w(\sigma)}{w(d_i(\sigma))}(-1)^id_i(\sigma)\] where the \emph{face maps} $d_i$ are defined as: \[d_i(\sigma)=[v_0,\dots,\widehat{v_i},\dots,v_n]\qquad\text{(deleting the vertex $v_i$)}\]  for any $n$-simplex $\sigma=[v_0,\dots,v_n]$. 

Again, if $w(\sigma)\neq 0$, then $w(d_i(\sigma))\neq 0$ so $\partial_n$ is well-defined. Similarly, we identify $\frac{w(\sigma)}{w(d_i(\sigma))}\in S^{-1}R$ with the corresponding element in $R$ as described in Remark \ref{identifyfrac}.
\end{defn}

Next we show that after generalization to weights in an integral domain, the relation $\partial^2=0$ (\cite[p.~234]{Dawson1990}) of the weighted boundary map remains true. 

\begin{prop}[cf.\ {\cite[p.~234]{Dawson1990}}]
$\partial^2=0$. To be precise, the composition $C_n(K)\xrightarrow{\partial_n}C_{n-1}(K)\xrightarrow{\partial_{n-1}}C_{n-2}(K)$ is the zero map.
\end{prop}
\begin{proof}
Let $\sigma=[v_0,\dots,v_n]$ be a $n$-simplex. We have \[\partial_n(\sigma)=\sum_{i=0}^n\frac{w(\sigma)}{w([v_0,\dots,\widehat{v_i},\dots,v_n])}(-1)^i[v_0,\dots,\widehat{v_i},\dots,v_n].\] Hence
\begin{align*}
&\partial_{n-1}\partial_n(\sigma)\\
&=\sum_{j<i}\frac{w(\sigma)}{w(d_i(\sigma))}(-1)^i\frac{w(d_i(\sigma))}{w([v_0,\dots,\widehat{v_j},\dots,\widehat{v_i},\dots,v_n])}(-1)^j[v_0,\dots,\widehat{v_j},\dots,\widehat{v_i},\dots,v_n]\\
&\quad+\sum_{j>i}\frac{w(\sigma)}{w(d_i(\sigma))}(-1)^i\frac{w(d_i(\sigma))}{w([v_0,\dots,\widehat{v_i},\dots,\widehat{v_j},\dots,v_n])}(-1)^{j-1}[v_0,\dots,\widehat{v_i},\dots,\widehat{v_j},\dots,v_n]\\
&=\sum_{j<i}\frac{w(\sigma)}{w([v_0,\dots,\widehat{v_j},\dots,\widehat{v_i},\dots,v_n])}(-1)^{i+j}[v_0,\dots,\widehat{v_j},\dots,\widehat{v_i},\dots,v_n]\\
&\quad+\sum_{j>i}\frac{w(\sigma)}{w([v_0,\dots,\widehat{v_i},\dots,\widehat{v_j},\dots,v_n])}(-1)^{i+j-1}[v_0,\dots,\widehat{v_i},\dots,\widehat{v_j},\dots,v_n]\\
&=0.
\end{align*}
The latter two summations cancel since after switching $i$ and $j$ in the second sum, it becomes the additive inverse of the first.
\end{proof}
\begin{lemma}
\label{fdcommute}
Let $f: K\to L$ be a simplicial map and $d_i$ be the $i$th face map. Then 
\begin{equation}d_i(f(\sigma))=f(d_i(\sigma))
\end{equation}
for all $\sigma=[v_0,v_1,\dots,v_n]\in K$ with $f(v_0),\dots,f(v_n)$ are distinct.
\end{lemma}
\begin{proof}
Let $\sigma=[v_0,\dots,v_n]$. Then we have
\begin{align*}
d_i(f(\sigma))&=d_i[f(v_0),\dots,f(v_n)]\\
&=[f(v_0),\dots,\widehat{f(v_i)},\dots,f(v_n)]\\
&=f([v_0,\dots,\widehat{v_i},\dots,v_n])\\
&=f(d_i(\sigma)).
\end{align*}
\end{proof}
\begin{prop}
Let $f:K\to L$ be a simplicial map. Then $f_\sharp\partial=\partial f_\sharp$.
\end{prop}
\begin{proof}
Let $\sigma=[v_0,\dots,v_n]\in C_n(K)$. Let $\tau$ be the simplex of $L$ spanned by $f(v_0),\dots,f(v_n)$. We consider three cases.
\begin{description}
\item[Case 1. $\dim \tau=n$]
In this case, the vertices $f(v_0),\dots,f(v_n)$ are distinct. We have
\begin{align*}
f_\sharp\partial(\sigma)&=f_\sharp\left(\sum_{i=0}^n\frac{w(\sigma)}{w(d_i(\sigma))}(-1)^id_i(\sigma)\right)\\
&=\sum_{i=0}^n\frac{w(\sigma)}{w(d_i(\sigma))}(-1)^i f_\sharp(d_i(\sigma))\\
&=\sum_{i=0}^n\frac{w(\sigma)}{w(d_i(\sigma))}(-1)^i\frac{w(d_i(\sigma))}{w(f(d_i(\sigma)))}f(d_i(\sigma))\\
&=\sum_{i=0}^n\frac{w(\sigma)}{w(f(d_i(\sigma)))}(-1)^if(d_i(\sigma)).
\end{align*}

On the other hand, we have
\begin{align*}
\partial f_\sharp(\sigma)&=\partial\left(\frac{w(\sigma)}{w(f(\sigma))}f(\sigma)\right)\\
&=\sum_{i=0}^n\frac{w(\sigma)}{w(f(\sigma))}\cdot\frac{w(f(\sigma))}{w(d_i(f(\sigma)))}(-1)^id_i(f(\sigma))\\
&=\sum_{i=0}^n\frac{w(\sigma)}{w(d_i(f(\sigma)))}(-1)^id_i(f(\sigma))\\
&=f_\sharp\partial(\sigma)
\end{align*}
since $d_i(f(\sigma))=f(d_i(\sigma))$ by Lemma \ref{fdcommute}.
\item[Case 2. $\dim\tau\leq n-2$]
In this case, $f_\sharp(d_i(\sigma))=0$ for all $i$, since at least two of the points $f(v_0),\dots,f(v_{i-1}),f(v_{i+1}),\dots,f(v_n)$ are the same. Thus $f_\sharp\partial(\sigma)$ vanishes. Note that $\partial f_\sharp(\sigma)$ also vanishes since $f_\sharp(\sigma)=0$, because $f(v_0),\dots f(v_n)$ are not distinct.
\item[Case 3. $\dim\tau=n-1$] WLOG we may assume that the vertices are ordered such that $f(v_0)=f(v_1)$, and $f(v_1),\dots,f(v_n)$ are distinct. Then $\partial f_\sharp(\sigma)$ vanishes. Now, \[f_\sharp\partial(\sigma)=\sum_{i=0}^n\frac{w(\sigma)}{w(d_i(\sigma))}(-1)^i f_\sharp(d_i(\sigma))\] has only two nonzero terms which sum up to
\begin{align*}
&\frac{w(\sigma)}{w(d_0(\sigma))}\cdot\frac{w(d_0(\sigma))}{w(f(d_0(\sigma)))}f(d_0(\sigma))-\frac{w(\sigma)}{w(d_1(\sigma))}\cdot\frac{w(d_1(\sigma))}{w(f(d_1(\sigma)))}f(d_1(\sigma))\\
=&\frac{w(\sigma)}{w(f(d_0(\sigma)))}f(d_0(\sigma))-\frac{w(\sigma)}{w(f(d_1(\sigma)))}f(d_1(\sigma)).
\end{align*}
Since $f(v_0)=f(v_1)$, we have $f(d_0(\sigma))=f(d_1(\sigma))$ and hence the two terms cancel each other as desired.
\end{description}
\end{proof}
\begin{defn}
We define the weighted homology group
\begin{equation}
H_n(K,w):=\ker(\partial_n)/\Ima(\partial_{n+1}),
\end{equation}
where $\partial_n$ is the weighted boundary map defined in Definition \ref{boundary}.
\end{defn}
Since the maps $f_\sharp:C_n(K,w_K)\to C_n(L,w_L)$ satisfy $f_\sharp\partial=\partial f_\sharp$, the $f_\sharp$'s define a chain map from the chain complex of $(K,w_K)$ to that of $(L,w_L)$. By Proposition \ref{chainmapinduce}, $f_\sharp$ induces a homomorphism $f_*: H_n(K,w_K)\to H_n(L,w_L)$. We may then view the map $(K,w_K)\mapsto H_n(K,w_K)$ as a functor $H_n: \textbf{WSC}\to \textbf{R-Mod}$ from the category of weighted simplicial complexes (\textbf{WSC}) to the category of $R$-modules (\textbf{R-Mod}).
\subsection{Calculation of Homology Groups in WSC}
The homology functor we define is different from the standard simplicial homology functor. For instance, it is possible for $H_0$ of a weighted simplicial complex to have torsion when the coefficient ring is $\mathbb{Z}$, as shown in \cite[p.~237]{Dawson1990}. We illustrate this more generally in the following example.
\begin{figure}[htbp]
\begin{center}
\begin{tikzpicture}[scale=0.8]
\draw (4.35,2.56) node[anchor=north west] {$x$};
\draw (6.62,4.72) node[anchor=north west] {$y$};
\draw (9.0,2.56) node[anchor=north west] {$z$};
\draw (5.,2.)-- (7.,4.);
\draw (9.,2.)-- (7.,4.);
\begin{scriptsize}
\draw [fill=black] (5.,2.) circle (2.5pt);
\draw [fill=black] (7.,4.) circle (2.5pt);
\draw [fill=black] (9.,2.) circle (2.5pt);
\end{scriptsize}
\end{tikzpicture}
\caption{Simplicial complex with 3 vertices $x$, $y$, $z$.}
\label{torsioneg}
\end{center}
\end{figure}
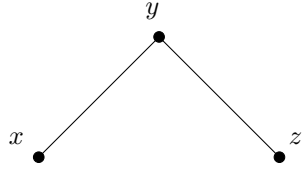
\begin{eg}[cf.\ {\cite[p.~237]{Dawson1990}}]
Let $R=\mathbb{Z}$. Consider $(K,w)$, where $w$ is the product weighting, to be the WSC shown in Figure \ref{torsioneg}, with $w(x)=1$, $w(y)=n$ and $w(z)=1$, where $n\in\mathbb{Z}$, $n\geq 2$. Then
\begin{align*}
\partial_1([x,y])&=\frac{w([x,y])}{w(y)}y-\frac{w([x,y])}{w(x)}x\\
&=\frac{n}{n}y-\frac{n}{1}x\\
&=y-nx.
\end{align*}

Similarly, $\partial_1([y,z])=nz-y$. Thus
\begin{align*}
H_0(K,w)&=\ker\partial_0/\Ima\partial_1\\
&\cong\langle x,y,z\mid nx=y, y=nz\rangle\\
&\cong\langle x,z\mid nx=nz\rangle\\
&\cong\mathbb{Z}\oplus\mathbb{Z}_n.
\end{align*}
\end{eg}
\begin{prop}[{cf.\ \cite[p.~239]{Dawson1990}}]
For the constant weighting $(K,a)$, $a\in R\setminus\{0\}$, the weighted homology functor is the same as the standard simplicial homology functor.
\end{prop}
\begin{proof}
If every simplex has weight $a\in R\setminus\{0\}$, note that the chain maps in Definition \ref{chain} and the weighted boundary maps in Definition \ref{boundary} reduce to the usual ones in standard simplicial homology. Hence the resulting weighted homology functor reduces to the standard one.
\end{proof}
\section{Weighted Persistent Homology}
After defining weighted homology, we proceed to define weighted persistent homology, following the example of the seminal paper by Zomorodian and Carlsson \cite{Zomorodian2005}. First, we give a review of persistence \cite{Zomorodian2005,Ghrist2008}, with generalizations to the weighted case.
\subsection{Persistence}
\begin{defn}
A \emph{weighted filtered complex} is an increasing sequence of weighted simplicial complexes $(\mathcal{K},w)=\{(K^i,w)\}_{i\geq 0}$, such that $K^i\subseteq K^{i+1}$ for all integers $i\geq 0$. (The weighting on $K^i$ is a restriction of that on $K^j$ for $i<j$.)
\end{defn}
Given a weighted filtered complex, for the $i$th complex $K^i$ we define the associated weighted boundary maps $\partial_k^i$ and groups $C_k^i, Z_k^i, B_k^i, H_k^i$ for all integers $i,k\geq 0$, following our development in Section \ref{homology}.
\begin{defn}
The \emph{weighted boundary map} $\partial_k^i: C_k(K^i)\to C_{k-1}(K^i)$ is the map $\partial_k: C_k(K^i)\to C_{k-1}(K^i)$ as defined in Definition \ref{boundary}. The \emph{weighted chain group} $C_k^i$ is the group $C_k(K^i,w)$ in Definition \ref{chaingroup}. The \emph{weighted cycle group} $Z_k^i$ is the group $\ker (\partial_k^i)$, while the \emph{weighted boundary group} $B_k^i$ is the group $\Ima(\partial_{k+1}^i)$. The \emph{weighted homology group} $H_k^i$ is the quotient group $Z_k^i/B_k^i$. (If the context is clear, we may omit the adjective ``weighted''.)
\end{defn}
\begin{defn}[{cf.\ \cite[p.~6]{Zomorodian2005}}]
The weighted \emph{$p$-persistent $k$th homology group} of $(\mathcal{K},w)=\{(K^i,w)\}_{i\geq 0}$ is defined as
\begin{equation}
H_k^{i,p}(\mathcal{K},w):=Z_k^i/(B_k^{i+p}\cap Z_k^i).
\end{equation}
If the coefficient ring $R$ is a PID and all the $K^i$ are finite simplicial complexes, then $H_k^{i,p}$ is a finitely generated module over a PID. We can then define the \emph{$p$-persistent $k$th Betti number} of $(K^i,w)$, denoted by $\beta_k^{i,p}$, to be the rank of the free submodule of $H_k^{i,p}$. This is well-defined by the structure theorem for finitely generated modules over a PID.
\end{defn}
Consider the homomorphism $\eta_k^{i,p}: H_k^i\to H_k^{i+p}$ that maps a homology class into the one that contains it. To be precise, 
\begin{equation}
n_k^{i,p}(\alpha+B_k^i)=\alpha+B_k^{i+p}.
\end{equation}
The homomorphism $\eta_k^{i,p}$ is well-defined since if $\alpha_1+B_k^i=\alpha_2+B_k^i$, then $\alpha_1-\alpha_2\in B_k^i\subseteq B_k^{i+p}$.
We prove that similar to the unweighted case (cf.\ \cite{Edelsbrunner2002,Zomorodian2005,Zomorodi2005}) we have $\Ima \eta_k^{i,p}\cong H_k^{i,p}$.
\begin{prop}[cf.\ {\cite[p.~6]{Zomorodian2005}}]
$\Ima \eta_k^{i,p}\cong H_k^{i,p}$.
\end{prop}
\begin{proof}
By the first isomorphism theorem, we have \[\Ima \eta_k^{i,p}\cong H_k^i/\ker \eta_k^{i,p}.\]
Note that
\begin{equation}
\begin{split}
&\alpha+B_k^i\in\ker\eta_k^{i,p}\\
&\iff \alpha+B_k^{i+p}=B_k^{i+p}\  \text{and}\ \alpha\in Z_k^i\\
&\iff\alpha\in B_k^{i+p}\cap Z_k^i\\
&\iff\alpha+B_k^i\in(B_k^{i+p}\cap Z_k^i)/B_k^i.
\end{split}
\end{equation} Hence \[\ker\eta_k^{i,p}=(B_k^{i+p}\cap Z_k^i)/B_k^i.\]
Hence we have
\begin{align*}
\Ima \eta_k^{i,p}&\cong H_k^i/\ker\eta_k^{i,p}\\
&=\frac{Z_k^i/B_k^i}{(B_k^{i+p}\cap Z_k^i)/B_k^i}\\
&\cong Z_k^i/(B_k^{i+p}\cap Z_k^i) \tag{by the third isomorphism theorem}\\
&=H_k^{i,p}.
\end{align*}
\end{proof}
\section{Applications}
Weighted persistent homology can tell apart filtrations that ordinary persistent homology does not distinguish. For instance, if there is a special point, weighted persistent homology can tell when a cycle containing the point is formed or has disappeared. This is a generalization of the main feature of persistent homology which is to detect the ``birth'' and ``death'' of cycles. We illustrate this in the following example.
\begin{eg}
\label{eg:finalfigure}
\begin{figure}[htbp]
\begin{subfigure}{0.2\textwidth}
\begin{tikzpicture}[scale=0.5]
\draw (1,3) node[anchor=east] {$v_0$};
\draw (1,1) node[anchor=east] {$v_1$};
\draw (3,3) node[anchor=west] {$v_3$};
\draw (3,1) node[anchor=west] {$v_2$};
\draw [fill=black] (1,3) circle (2.5pt);
\draw [fill=black] (1,1) circle (2.5pt);
\draw [fill=black] (3,3) circle (2.5pt);
\draw [fill=black] (3,1) circle (2.5pt);
\end{tikzpicture}
\caption{$K^0$}
\end{subfigure}
\begin{subfigure}{0.2\textwidth}
\begin{tikzpicture}[scale=0.5]
\draw (1,3) node[anchor=east] {$v_0$};
\draw (1,1) node[anchor=east] {$v_1$};
\draw (3,3) node[anchor=west] {$v_3$};
\draw (3,1) node[anchor=west] {$v_2$};
\draw [fill=black] (1,3) circle (2.5pt);
\draw [fill=black] (1,1) circle (2.5pt);
\draw [fill=black] (3,3) circle (2.5pt);
\draw [fill=black] (3,1) circle (2.5pt);
\draw (1,3)--(1,1)--(3,3)--cycle;
\end{tikzpicture}
\caption{$K^1$}
\end{subfigure}
\begin{subfigure}{0.2\textwidth}
\begin{tikzpicture}[scale=0.5]
\draw (1,3) node[anchor=east] {$v_0$};
\draw (1,1) node[anchor=east] {$v_1$};
\draw (3,3) node[anchor=west] {$v_3$};
\draw (3,1) node[anchor=west] {$v_2$};
\draw [fill=black] (1,3) circle (2.5pt);
\draw [fill=black] (1,1) circle (2.5pt);
\draw [fill=black] (3,3) circle (2.5pt);
\draw [fill=black] (3,1) circle (2.5pt);
\draw (1,3)--(1,1)--(3,3)--cycle;
\draw (3,3)--(3,1);
\end{tikzpicture}
\caption{$K^2$}
\end{subfigure}
\begin{subfigure}{0.2\textwidth}
\begin{tikzpicture}[scale=0.5]
\draw (1,3) node[anchor=east] {$v_0$};
\draw (1,1) node[anchor=east] {$v_1$};
\draw (3,3) node[anchor=west] {$v_3$};
\draw (3,1) node[anchor=west] {$v_2$};
\draw [fill=black] (1,3) circle (2.5pt);
\draw [fill=black] (1,1) circle (2.5pt);
\draw [fill=black] (3,3) circle (2.5pt);
\draw [fill=black] (3,1) circle (2.5pt);
\filldraw [fill=gray] (1,3)--(1,1)--(3,3)--cycle;
\draw (3,3)--(3,1);
\end{tikzpicture}
\caption{$K^3$}
\end{subfigure}
\caption{The filtration $\mathcal{K}=\{K^0,K^1,K^2,K^3\}$, where the shaded region denotes the 2-simplex $[v_0,v_1,v_3]$.}
\label{fig:filtK}
\end{figure}
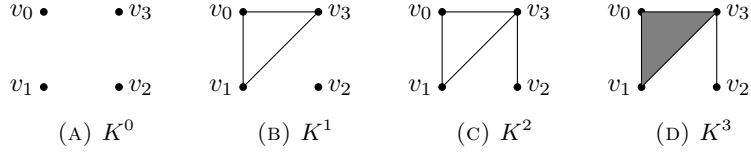
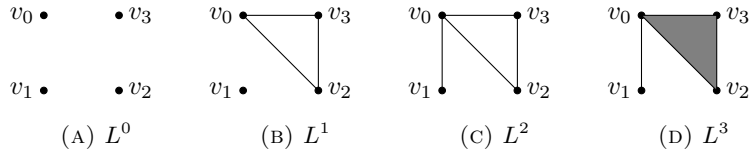
\begin{figure}[htbp]
\begin{subfigure}{0.2\textwidth}
\begin{tikzpicture}[scale=0.5]
\draw (1,3) node[anchor=east] {$v_0$};
\draw (1,1) node[anchor=east] {$v_1$};
\draw (3,3) node[anchor=west] {$v_3$};
\draw (3,1) node[anchor=west] {$v_2$};
\draw [fill=black] (1,3) circle (2.5pt);
\draw [fill=black] (1,1) circle (2.5pt);
\draw [fill=black] (3,3) circle (2.5pt);
\draw [fill=black] (3,1) circle (2.5pt);
\end{tikzpicture}
\caption{$L^0$}
\end{subfigure}
\begin{subfigure}{0.2\textwidth}
\begin{tikzpicture}[scale=0.5]
\draw (1,3) node[anchor=east] {$v_0$};
\draw (1,1) node[anchor=east] {$v_1$};
\draw (3,3) node[anchor=west] {$v_3$};
\draw (3,1) node[anchor=west] {$v_2$};
\draw [fill=black] (1,3) circle (2.5pt);
\draw [fill=black] (1,1) circle (2.5pt);
\draw [fill=black] (3,3) circle (2.5pt);
\draw [fill=black] (3,1) circle (2.5pt);
\draw (1,3)--(3,3)--(3,1)--cycle;
\end{tikzpicture}
\caption{$L^1$}
\end{subfigure}
\begin{subfigure}{0.2\textwidth}
\begin{tikzpicture}[scale=0.5]
\draw (1,3) node[anchor=east] {$v_0$};
\draw (1,1) node[anchor=east] {$v_1$};
\draw (3,3) node[anchor=west] {$v_3$};
\draw (3,1) node[anchor=west] {$v_2$};
\draw [fill=black] (1,3) circle (2.5pt);
\draw [fill=black] (1,1) circle (2.5pt);
\draw [fill=black] (3,3) circle (2.5pt);
\draw [fill=black] (3,1) circle (2.5pt);
\draw (1,3)--(3,3)--(3,1)--cycle;
\draw (1,3)--(1,1);
\end{tikzpicture}
\caption{$L^2$}
\end{subfigure}
\begin{subfigure}{0.2\textwidth}
\begin{tikzpicture}[scale=0.5]
\draw (1,3) node[anchor=east] {$v_0$};
\draw (1,1) node[anchor=east] {$v_1$};
\draw (3,3) node[anchor=west] {$v_3$};
\draw (3,1) node[anchor=west] {$v_2$};
\draw [fill=black] (1,3) circle (2.5pt);
\draw [fill=black] (1,1) circle (2.5pt);
\draw [fill=black] (3,3) circle (2.5pt);
\draw [fill=black] (3,1) circle (2.5pt);
\filldraw [fill=gray] (1,3)--(3,3)--(3,1)--cycle;
\draw (1,3)--(1,1);
\end{tikzpicture}
\caption{$L^3$}
\end{subfigure}
\caption{The filtration $\mathcal{L}=\{L^0,L^1,L^2,L^3\}$, where the shaded region denotes the 2-simplex $[v_0,v_2,v_3]$.}
\label{fig:filtL}
\end{figure}
Consider the two filtrations as shown in Figure \ref{fig:filtK} and \ref{fig:filtL}. By symmetry, it is clear that the (unweighted) persistent homology groups of the two filtrations will be the same.

Suppose we consider $v_2$ as a special point and wish to tell through weighted persistent homology whether a 1-cycle containing $v_2$ is formed or has disappeared. We can achieve it by the following weight function (choosing $R=\mathbb{Z}$). Let $w$ be the weight function such that all 2-dimensional (and higher) simplices containing $v_2$ have weight 2, while all other simplices have weight 1. In our example, this means $w([v_0,v_2,v_3])=2$ while $w(\sigma)=1$ for all $\sigma\neq [v_0,v_2,v_3]$. 

Then for the filtration $\mathcal{K}=\{K^0,K^1,K^2,K^3\}$ we have
\begin{align*}
Z_1^1&=\ker(\partial_1^1)=\langle[v_0,v_1]-[v_0,v_3]+[v_1,v_3]\rangle\\
\partial_2^3([v_0,v_1,v_3])&=[v_1,v_3]-[v_0,v_3]+[v_0,v_1]\\
\partial_2^1&=\partial_2^2=0.
\end{align*}
Hence, we have \[H_1^{1,p}(\mathcal{K},w)=\begin{cases}\mathbb{Z} &\text{for}\ p=0,1\\
0 &\text{for}\ p=2.
\end{cases}\]

However for the filtration $\mathcal{L}=\{L^0,L^1,L^2,L^3\}$ we have
\begin{align*}
Z_1^1&=\ker(\partial_1^1)=\langle [v_2,v_3]-[v_0,v_3]+[v_0,v_2]\rangle\\
\partial_2^3([v_0,v_2,v_3])&=2[v_2,v_3]-2[v_0,v_3]+2[v_0,v_2]\\
\partial_2^1&=\partial_2^2=0
\end{align*}
so that
\begin{equation}
\label{torsionegeqn}
H_1^{1,p}(\mathcal{L},w)=\begin{cases}
\mathbb{Z} &\text{for}\ p=0,1\\
\mathbb{Z}_2 &\text{for}\ p=2.
\end{cases}
\end{equation}
\end{eg}
Referring to Equation (\ref{torsionegeqn}), we can interpret the presence of torsion in $H_1^{1,2}(\mathcal{L},w)$ to mean that a 1-cycle containing $v_2$ is formed in $L^1$, persists in $L^2$, and disappears in $L^3$.
\begin{remark}
Let $R=\mathbb{Z}$. Generalizing Example \ref{eg:finalfigure}, if there is a special point $v$, we can tell if a $k$-cycle containing $v$ is formed or has disappeared by setting all $k+1$-dimensional and higher simplices containing $v$ to have weight $m\geq 2$, and all other simplices to have weight 1.
\end{remark}

\subsection{Algorithm for PIDs}
For coefficients in a PID, we show that the weighted persistent homology groups are computable. In the seminal paper \cite{Zomorodian2005} by Zomorodian and Carlsson, the authors show an algorithm for persistent homology over a PID. We present an algorithm in this section, which is a weighted modification of the algorithm in \cite{Zomorodian2005} based on the reduction algorithm. We use Figure \ref{fig:filtL} as a running example to illustrate the algorithm.

Let $R$ be a PID. We represent the weighted boundary operator $\partial_n: C_n(K,w)\to C_{n-1}(K,w)$ relative to the standard bases (The standard basis for $C_n(K,w)$ is the set of $n$-simplices of $K$ with nonzero weight (see Definition \ref{chaingroup})) of the respective weighted chain groups as a matrix $M_n$ with entries in $R$. The matrix $M_n$ is called the \emph{standard matrix representation} of $\partial_n$. It has $m_n$ columns and $m_{n-1}$ rows, where $m_n$, $m_{n-1}$ are the number of $n$- and $(n-1)$-simplices with nonzero weights respectively.

In general, due to the weights, the matrix $M_n$ for the weighted boundary map is \emph{different} from that of the unweighted case. For instance, for the unweighted case the matrix representation is restricted to having entries in $\{-1_R,0_R,1_R\}$, while the weighted matrix representation can have entries taking arbitrary values in the ring $R$. In particular, when performing the reduction algorithm, we need to make the modification to allow the following \emph{elementary row operations} on $M_k$:
\begin{enumerate}
\item exchange row $i$ and row $j$,
\item multiply row $i$ by a unit $u\in R\setminus\{0\}$,
\item replace row $i$ by (row $i$)+$q$(row $j$), where $q\in R\setminus\{0\}$ and $j\neq i$.
\end{enumerate}
Note that for the unweighted case \cite[p.~5]{Zomorodian2005}, the second elementary row operation was ``multiply row $i$ by $-1$''. A similar modification is also needed for the \emph{elementary column operations}.

The subsequent steps are similar to that of the unweighted case (cf. \cite[pp.~5,12]{Zomorodian2005}). We summarize the algorithm (Algorithm \ref{alg:reduction}) and refer the reader to \cite[p.~5]{Zomorodian2005} for more information on the reduction algorithm and the Smith normal form.

Given a weighted filtered complex $\{(K^i,w)\}_{i\geq0}$, we write $M_k^i$ to denote the standard matrix representation of $\partial_k^i$. We perform the Algorithm \ref{alg:reduction} to obtain the weighted homology groups.

\begin{algorithm}
\caption{Weighted Persistent Homology Algorithm for PIDs (cf.\ {\cite[p.~12]{Zomorodian2005}})}
\label{alg:reduction}
\begin{flushleft}
\textbf{Input}: Weighted filtered complex $(\mathcal{K},w)=\{(K^i,w)\}_{i\geq0}$

\textbf{Output}: Weighted $p$-persistent $k$th homology group $H_k^{i,p}(\mathcal{K},w)$
\end{flushleft}
\begin{enumerate}
\item Reduce the matrix $M_k^i$ to its Smith normal form and obtain a basis $\{z^j\}$ for $Z_k^i$.
\item Reduce the matrix $M_{k+1}^{i+p}$ to its Smith normal form and obtain a basis $\{b^l\}$ for $B_k^{i+p}$.
\item Let $A=[\{b^l\}\;\{z^j\}]=[B\;Z]$, i.e.\ the columns of matrix $A$ consist of the basis elements computed in the previous steps, with respect to the standard basis of $C_k(K^{i+p},w)$. We reduce $A$ to its Smith normal form to find a basis $\{a^q\}$ for its nullspace. 
\item Each $a^q=[\alpha^q\;\beta^q]$, where $\alpha^q$, $\beta^q$ are column vectors of coefficients of $\{b^l\}$, $\{z^j\}$ respectively. Since $Au^q=B\alpha^q+Z\beta^q=0$, the element $\beta\alpha^q=-Z\beta^q$ belongs to the span of both bases $\{z^j\}$ and $\{b^l\}$. Hence, both $\{B\alpha^q\}$ and $\{Z\beta^q\}$ are bases for $B_k^{i,p}=B_k^{i+p}\cap Z_k^i$. Using either, we form the matrix $M_{k+1}^{i,p}$ using the basis. The number of columns of $M_{k+1}^{i,p}$ is the cardinality of the basis for $B_k^{i,p}$, while the number of rows is the cardinality of the standard basis for $C_k(K^{i+p},w)$.
\item We reduce $M_{k+1}^{i,p}$ to Smith normal form to read off the torsion coefficients of $H_k^{i,p}(\mathcal{K},w)$ and the rank of $B_k^{i,p}$.
\item The rank of the free submodule of $H_k^{i,p}(\mathcal{K},w)$ is the rank of $Z_k^i$ minus the rank of $B_k^{i,p}$.
\end{enumerate}
 \end{algorithm}
 
 We illustrate the algorithm using Example \ref{eg:reduction}.
 \begin{eg}
 \label{eg:reduction}
 Consider the filtration $\mathcal{L}=\{L^0,L^1,L^2,L^3\}$ in Figure \ref{fig:filtL}. We have
 \begin{equation}
 \begin{split}
 M^1_1&=\mleft[\begin{array}{c|ccc}
 &[v_0,v_3] &[v_0,v_2] &[v_2,v_3]\\
    \hline
v_0 &-1 &-1 &0\\
v_1 & 0 & 0 & 0\\
v_2 & 0 & 1 & -1\\
v_3 & 1 & 0 & 1
 \end{array}
\mright]\\
&\xrightarrow{\text{reduce}}\mleft[\begin{array}{c|ccc}
 &[v_0,v_3] &[v_0,v_2] &[v_2,v_3]-[v_0,v_3]+[v_0,v_2]\\
    \hline
v_3-v_0 &1 &0 &0\\
v_2-v_0 & 0 &1 &0\\
v_1 & 0 &0 &0\\
v_2 & 0 & 0 & 0
 \end{array}
\mright].
\end{split}
 \end{equation}
 
 Hence a basis for $Z_1^1$ is $\{[v_2,v_3]-[v_0,v_3]+[v_0,v_2]\}$.

 \begin{equation}
 \begin{split}
M_2^3&=\mleft[\begin{array}{c|c}
&[v_0,v_2,v_3]\\
\hline
[v_0,v_1] &0\\
{[}v_0,v_2] &2\\
{[}v_0,v_3] &-2\\
{[}v_2,v_3] &2
\end{array}
\mright]\\
&\xrightarrow{\text{reduce}}\mleft[\begin{array}{c|c}
&[v_0,v_2,v_3]\\
\hline
[v_0,v_2]-[v_0,v_3]+[v_2,v_3] &2\\
{[}v_0,v_3] &0\\
{[}v_2,v_3] &0\\
{[}v_0,v_1] &0
\end{array}
\mright]
\end{split}
\end{equation}

Hence a basis for $B_1^3$ is $\{2[v_0,v_2]-2[v_0,v_3]+2[v_2,v_3]\}$. Let $b=2[v_0,v_2]-2[v_0,v_3]+2[v_2,v_3]$ and $z=[v_0,v_2]-[v_0,v_3]+[v_2,v_3]$.

\begin{equation}
A=[B\;Z]\\=\mleft[\begin{array}{c|cc}
&b &z\\
\hline
{[}v_0,v_1] &0 &0\\
{[}v_0,v_2] &2 &1\\
{[}v_0,v_3] &-2 &-1\\
{[}v_2,v_3] &2 &1
\end{array}\mright]\\
\xrightarrow{reduce}\mleft[\begin{array}{c|cc}
&z &b-2z\\
\hline
z & 1 &0\\
{[}v_0,v_3] &0 &0\\
{[}v_2,v_3] &0 &0\\
{[}v_0,v_1] &0 &0
\end{array}\mright]
\end{equation}

Hence a basis for the nullspace of $A$ is $\{b-2z\}$. In this context, a basis for $B_1^{1,2}$ is $\{B\alpha^q\}=\{b\}$. Hence we form a matrix
\begin{equation}
\label{eq:torsionmatrix}
M_2^{1,2}=\mleft[
\begin{array}{c|c}
&b\\
\hline
{[}v_0,v_1] &0\\
{[}v_0,v_2] &2\\
{[}v_0,v_3] &-2\\
{[v}_2,v_3] &2
\end{array}
\mright]\\
\xrightarrow{\text{reduce}}\mleft[
\begin{array}{c|c}
&b\\
\hline
z &2\\
{[}v_0,v_1] &0\\
{[}v_0,v_3] &0\\
{[}v_2,v_3] &0
\end{array}
\mright].
\end{equation} 
Since both $Z_1^1$ and $B_1^{1,2}$ have rank 1, the rank of the free part of $H_1^{1,2}(\mathcal{L},w)$ is $1-1=0$. We read off (\ref{eq:torsionmatrix}) and conclude that $H_1^{1,2}(\mathcal{L},w)=\mathbb{Z}_2$, which agrees with our previous computation in Example \ref{eg:finalfigure}.
\end{eg}
\bibliographystyle{amsplain}
\bibliography{jabref5}

\end{document}